\documentclass[11pt]{article}
\usepackage{amsmath,amsthm,amssymb}
\usepackage{color}

\def\titlerunning#1{\gdef\titrun{#1}}
\makeatletter
\def\author#1{\gdef\autrun{\def\and{\unskip, }#1}\gdef\@author{#1}}
\def\address#1{{\def\and{\\\hspace*{18pt}}\renewcommand{\thefootnote}{}%
\footnote {#1}}%
\markboth{\autrun}{\titrun}}
\makeatother
\def\email#1{\hspace*{4pt}{\em e-mail}: #1}
\def\MSC#1{{\renewcommand{\thefootnote}{}%
\footnote{\emph{Mathematics Subject Classification (2010):} #1}}}
\def\keywords#1{\par\medskip
\noindent\textbf{Keywords:} #1}

\newtheorem{theorem}{Theorem}[section]
\newtheorem{prop}[theorem]{Proposition}
\newtheorem{cor}[theorem]{Corollary}
\newtheorem{lemma}[theorem]{Lemma}

\newtheorem{defin}[theorem]{Definition}

\theoremstyle{definition}

\numberwithin{equation}{section}

\def\cC{\mathcal C}
\def\cB{\mathcal B}
\def\cE{\mathcal E}

\def\cH{\mathcal H}
\def\cI{\mathcal I}

\def\cO{\mathcal O}
\def\cP{\mathcal P}

\def\cV{\mathcal V}

\def\cW{\mathcal W}
\def\cQ{\mathcal Q}
\def\PG{{\rm PG}}
\def\PGaSp{{\rm P\Gamma Sp}}

\def\PGaU{{\rm P\Gamma U}}

\def\PSU{{\rm PSU}}
\def\PSp{{\rm PSp}}

\begin{document}


\baselineskip=16pt

\titlerunning{}

\title{On the $\PSU(4,2)$--invariant vertex--transitive strongly regular $(216, 40, 4, 8)$ graph}


\author{Dean Crnkovi\' c, Francesco Pavese and Andrea \v{S}vob }

\date{}

\maketitle

\address{D. Crnkovi\' c, A. \v{S}vob: Department of Mathematics, University of Rijeka, Radmile Matej\v{c}i\'c 2, 51000 Rijeka, Croatia;
\email{\{deanc,asvob\}@math.uniri.hr}
\and
F. Pavese: Dipartimento di Meccanica Matematica e Management, Politecnico di Bari, Via Orabona 4,  I-70125 Bari, Italy;
\email{francesco.pavese@poliba.it} 
}

\bigskip

\MSC{Primary 51E20; Secondary 05E30, 51E12, 51A50.}


\begin{abstract}
In 2018 the first, Rukavina and the third author constructed with the aid of a computer the first example of a strongly regular graph $\Gamma$ with parameters $(216, 40, 4, 8)$ and proved that it is the unique $\PSU(4,2)$--invariant vertex--transitive graph on $216$ vertices. In this paper, using the geometry of the Hermitian surface of $\PG(3,4)$, we provide a computer--free proof of the existence of the graph $\Gamma$. The maximal cliques of $\Gamma$ are also determined. 

\keywords{strongly regular graph; Hermitian surface; generalized quadrangle; projective geometry.}
\end{abstract}

\section{Introduction}

A {\em strongly regular graph} with parameters $(v, k, \lambda, \mu)$ is a (simple, undirected, connected) graph with $v$ vertices such that each vertex lies on $k$ edges, any two adjacent vertices have exactly $\lambda$ common neighbours, and any two non--adjacent vertices have exactly $\mu$ common neighbours.

A {\em generalized quadrangle of order} $(s,t)$ (GQ$(s,t)$ for short) is an incidence structure $(\cP,\cB,\cI)$ of points and lines with the properties that any two points (lines) are incident with at most one line (point), every point is incident with $t+1$ lines, every line is incident with $s+1$ points, and for any point $P$ and line $\ell$ which are not incident, there is a unique point on $\ell$ collinear with $P$. An {\em ovoid} $\cO$ of a GQ$(s, t)$ is a subset of $\cP$ such that each line of $\cB$ is incident with a unique point of $\cO$. For more information we refer the readers to \cite{PT}.

There is a unique GQ$(2,2)$ and a unique GQ$(4,2)$. These are the incidence structures of points and lines of the symplectic polar space $\cW(3,2)$ of $\PG(3, 2)$ and of the Hermitian surface $\cH(3, 4)$ of $\PG(3, 4)$, respectively. Moreover, $\cW(3, 2)$ can be embedded in $\cH(3, 4)$. In particular, $\cH(3, 4)$ contains $36$ symplectic subquadrangles and two of them share either $3$ or $6$ points. An ovoid of $\cW(3, 2)$ is an elliptic quadric $\cQ^-(3, 2)$ and a symplectic quadrangle $\cW(3, 2)$ possesses $6$ elliptic ovoids. 

With the aid of a computer, in \cite{CRS} the authors proved the existence of a strongly regular graph with parameters $(216, 40, 4, 8)$. The constructed strongly regular graph admits a vertex--transitive action of the group $\PSU(4,2)$. Using the geometry of the Hermitian surface of $\PG(3,4)$, in this paper we provide a computer--free proof of the existence of the $\PSU(4,2)$--invariant vertex--transitive strongly regular graph $\Gamma$ with parameters $(216, 40, 4, 8)$ firstly constructed in \cite{CRS}. The vertices of the graph $\Gamma$ are the elliptic ovoids of the symplectic subquadrangles of $\cH(3, 4)$ and two vertices are adjacent whenever the elliptic quadrics meet in exactly one point and their symplectic subquadrangles meet in six points. In the last part of the paper, the maximal cliques of the graph $\Gamma$ are also determined.

\section{The geometry of the Hermitian surface of $\PG(3,4)$}\label{setting}

In this section we provide a description of some of the properties of the Hermitian surface of $\PG(3, 4)$. For further information on the topic we refer the reader to \cite{H}.

Let $\cH(3, 4)$ be a Hermitian surface of $\PG(3, 4)$, i.e, the incidence structure of all totally isotropic points and totally isotropic lines (called {\em generators}) with respect to a (non--degenerate) unitary polarity of $\PG(3, 4)$. It is a generalized quadrangle of order $(4, 2)$ and its automorphism group $G$ is isomorphic to $\PGaU(4,q^2)$. Let $\perp$ be the non--degenerate unitary polarity of $\PG(3, 4)$ which gives rise to $\cH(3, 4)$. The Hermitian surface $\cH(3, 4)$ contains $45$ points. Any line of $\PG(3, 4)$ meets $\cH(3, 4)$ in either $1$ or $3$ or $5$ points. The latter lines are the {\em generators} of $\cH(3, 4)$, and their number is $27$. The lines meeting $\cH(3, 4)$ in $3$ points are called {\em secant} lines, whereas the lines meeting $\cH(3, 4)$ in one point are called {\em tangent lines}. Through every point $P$ of $\cH(3, 4)$ there pass exactly $3$ generators, and these generators are coplanar. The plane containing these generators is the polar plane $P^{\perp}$ of $P$. The tangent lines through $P$ are precisely the remaining $2$ lines of $P^{\perp}$ incident with $P$. In this case, the plane $P^{\perp}$ is also called the {\em tangent} plane to $\cH(3, 4)$ at $P$. If $P \notin \cH(3, 4)$ then $P^{\perp}$ is a plane of $\PG(3, 4)$ meeting $\cH(3, 4)$ in a non--degenerate Hermitian curve. In this case, the plane $P^{\perp}$ is also said to be {\em secant} to $\cH(3, 4)$.

Let $\Sigma$ be a subgeometry of $\PG(3, 4)$ isomorphic to $\PG(3, 2)$ and let $\cW(3, 2)$ be the incidence structure of all totally isotropic points and totally isotropic lines (called {\em generators}) with respect to a (non--degenerate) symplectic polarity of $\Sigma$. It is a generalized quadrangle of order $(2, 2)$ with automorphism group $\PGaSp(4,q)$. Then, it consists of all the points of $\PG(3, 2)$ and of $15$ generators. Through every point $P \in \Sigma$ there pass $3$ generators and these lines are coplanar. The plane containing these lines is the polar plane of $P$ with respect to the symplectic polarity which defines $\cW(3, 2)$. Such a polarity extends uniquely to a symplectic polarity $\tilde{\perp}$ of $\PG(3, 4)$. The symplectic quadrangle $\cW(3, 2)$ is a subquadrangle of $\cH(3, 4)$ if $\tilde{\perp} \perp = \perp \tilde{\perp} = \sigma$. In this case the two polarities are said to be {\em commuting} and $\sigma$ is a semilinear involution of $\PG(3, 4)$ stabilizing $\cH(3, 4)$ and the set of points fixed pointwise by $\sigma$ coincides with $\Sigma$. In other word, $\perp$ and $\tilde{\perp}$ induce the same polarity on $\Sigma$. On the other hand, any Baer subgeometry of $\PG(3, 4)$ whose points are contained in ${\cH}(3, 4)$ arises from a symplectic polarity of $\PG(3, 4)$ commuting with $\perp$. Let $\ell$ be a generator of $\cH(3, 4)$, then either $\ell \cap \Sigma$ is a generator of $\cW(3, 2)$ or $|\ell \cap \Sigma| = 0$. The stabilizer in $G$ of a subquadrangle $\cW(3, 2)$ of $\cH(3, 4)$ is a group $G_1$, isomorphic to $\PSp(4, 2) \rtimes \langle \sigma \rangle \simeq S_6 \rtimes \langle \sigma \rangle$. 

In the dual setting, $\cH(3, 4)$ corresponds to $\cQ^-(5,2)$, and a symplectic subquadrangles of $\cH(3, 4)$ is a parabolic quadric $\cQ(4,2)$ obtained by intersecting $\cQ^-(5, 2)$ with a non--degenerate hyperplane section. Two parabolic quadrics of $\cQ^-(5, 2)$ meet in either a three--dimensional hyperbolic quadric $\cQ^+(3, 2)$ or a quadratic cone. It follows that there are $36$ symplectic subquadrangles embedded in $\cH(3, 4)$ and two distinct symplectic subquadrangles of $\cH(3, 4)$ meet in either $3$ or $6$ points. In the former case the points of intersection are on a generator of $\cH(3, 4)$, while in the latter case are those of $\cH(3, 4)$ on two secant lines $r$ and $r^{\perp}$. If $r$ is a line that is secant to $\cH(3, 4)$, then there are $3$ symplectic subquadrangles of $\cH(3, 4)$ such that they pairwise meet in the six points of $(r \cup r^\perp) \cap \cH(3, 4)$. 

\begin{lemma}\label{lemma1}
Let $\cW$, $\cW'$ be two symplectic subquadrangles of $\cH(3, 4)$. 
\begin{itemize}
\item[i)] If $|\cW \cap \cW'| = 3$, then there are twelve symplectic subquadrangles of $\cH(3, 4)$ meeting both $\cW$, $\cW'$ in six points.
\item[ii)] If $|\cW \cap \cW'| = 6$, then there are ten symplectic subquadrangles of $\cH(3, 4)$ meeting both $\cW$, $\cW'$ in six points.   
\end{itemize} 
\end{lemma}
\begin{proof}
In the dual setting, $\cW$, $\cW'$ correspond to two parabolic quadrics $\cQ(4,2)$, say $\cP$, $\cP'$. Moreover, a symplectic subquadrangle meeting both $\cW$, $\cW'$ in six points corresponds to a further parabolic quadric $\tilde{\cP}$ of $\cQ^-(5, 2)$ such that $\cP \cap \tilde{\cP}$ and $\cP' \cap \tilde{\cP}$ are three--dimensional hyperbolic quadrics $\cQ^+(3, 2)$.

If $|\cW \cap \cW'| = 3$, then $\cP \cap \cP'$ is a quadratic cone of a three--dimensional space $\Pi$. The points of $\cQ^-(5,2)~\setminus~\left(\cP~\cup~\cP'~\setminus~\left(\cP~\cap~\cP'\right)\right)$ are those in the tangent hyperplane $H$ to $\cQ^-(5,2)$ at the vertex of the cone. Note that $\cP \cap \tilde{\cP} \cap \Pi = \cP' \cap \tilde{\cP} \cap \Pi = \tilde{\cP} \cap \Pi$ is either a conic or contains two of the three lines of $\cP \cap \cP'$. However, the former case cannot occur since in a $\cQ(4, 2)$ through a conic there pass at most one hyperbolic quadric. On the other hand, in $H$ there are six quadratic cones meeting $\cP \cap \cP'$ in two lines and through each of these cones there are two parabolic quadrics that necessarily meet both $\cP$ and $\cP'$ in a hyperbolic quadric, as required. 

If $|\cW \cap \cW'| = 6$, then $\cP \cap \cP'$ is a three--dimensional hyperbolic quadric. As before, $\tilde{\cP}$ cannot intersect $\cP \cap \cP'$ in a conic. Note that there is a third parabolic quadric meeting both $\cP$, $\cP'$ exactly in $\cP \cap \cP'$. Moreover, if $P$ is a point of $\cP \cap \cP'$, then there are two lines of $\cP \cap \cP'$ intersecting in $P$ and generating a plane $\sigma_{P}$. There are two three--dimensional hyperbolic quadrics, say $Q$ and $Q'$, distinct from $\cP \cap \cP'$, containing the two lines of $\sigma_{P}$ and contained in $\cP$ and in $\cP'$, respectively. The four--dimensional space $\langle Q, Q' \rangle$ meets $\cQ^-(5,2)$ in a parabolic quadric. Varying the point $P$ in $\cP \cap \cP'$ the result follows. 
\end{proof}

An elliptic ovoid of $\cW(3, 2)$ is an elliptic quadric $\cQ^-(3, 2)$ of $\Sigma$ such that the lines of $\Sigma$ that are tangent to $\cQ^-(3, 2)$ are precisely the $15$ generators of $\cW(3, 2)$. The stabilizer in $G_1$ (or in $G$) of an elliptic ovoid is a group $G_2$ isomorphic to $S_5 \rtimes \langle \sigma \rangle$. There are $6$ elliptic ovoids of $\cW(3, 2)$ and any two of them have precisely one point in common. An elliptic quadric $\cQ^-(3, 2)$ of a Baer subgeometry $\Sigma$ of $\PG(3, 4)$ defines a unique symplectic quadrangle $\cW(3, 2)$ of $\Sigma$. In this case we will say that $\cW(3, 2)$ is the {\em symplectic quadrangle of} $\cQ^-(3, 2)$. The group $G_1$ acts transitively on the $60$ non--ordered couples of distinct points of $\Sigma$ such that the line joining them is not a generator of $\cW(3, 2)$ and through two such points there is exactly one elliptic ovoid of $\cW(3,2)$. 

\section{A strongly regular graph with parameters $(216, 40, 4, 8)$}

Let $\cV$ be the set of elliptic ovoids of the symplectic subquadrangles of $\cH(3, 4)$. From the discussion above it follows that $|\cV| = 216$. 
\begin{defin}
Let $\Gamma$ be the graph whose vertices are the elliptic quadrics of $\cV$ and two vertices are adjacent whenever the corresponding elliptic quadrics meet in exactly one point and their symplectic subquadrangles meet in six points.
\end{defin}

Let $\cE, \cE'$ be two elliptic quadrics of $\cV$ and let $\cW, \cW'$ be their symplectic subquadrangles. Let $\Sigma, \Sigma'$ denote the Baer subgeometries containing $\cW$ and $\cW'$, respectively. If the subquadrangles $\cW, \cW'$ share exactly six points lying on two lines, say $s$ and $s^\perp$, then there is a third subquadrangle of $\cH(3, 4)$, say $\cW''$, contained in the Baer subgeometry $\Sigma''$, which intersects both $\cW$ and $\cW'$ in the six points of $(s \cup s^\perp) \cap \cH(3, 4)$. Then one of the following possibilities occur:
\begin{itemize}
\item[{\em i)}] $\cE$ and $\cE'$ are adjacent and $\cE \cap \cE'$ is a point. 
\item[{\em ii)}]  $\cE$ and $\cE'$ are not adjacent and $|\cE \cap \cE'| \in \{0, 2\}$.  
\end{itemize}
If the subquadrangles $\cW, \cW'$ have in common three points of a generator, then $\cE$ and $\cE'$ are not adjacent and $|\cE \cap \cE'| \in \{0, 1\}$. Finally, if $\cW = \cW'$ then $\cE$ and $\cE'$ are not adjacent and $|\cE \cap \cE'| = 1$.   

\begin{prop}\label{adj}
Let $\cE, \cE' \in \cV$. If $\cE$ and $\cE'$ are adjacent, then there are four elliptic quadrics of $\cV$ adjacent to both $\cE$ and $\cE'$.
\end{prop}
\begin{proof}
Suppose that $\cE, \cE'$ are two adjacent elliptic quadrics of $\cV$. Then $\cE \cap \cE'$ is a point and the subquadrangles $\cW, \cW'$ share exactly six points $(s \cup s^\perp) \cap \cH(3, 4)$. Note that $\{s, s^\perp\}$ consists of one secant line containing $\cE \cap \cE'$ and one external line to both $\cE, \cE'$. 

Let $s \cap \Sigma = \{R, S, T\}$, $\cE \cap s = \{R, S\}$ and $\cE' \cap s  = \{S, T\}$. Let $\tilde{\cE}$ be an elliptic ovoid of a symplectic subquadrangle $\tilde{\cW}$ of $\cH(3, 4)$ adjacent to both $\cE, \cE'$. Let $\tilde{\Sigma}$ be the Baer subgeometry containing $\tilde{\cW}$ and let $\cW \cap \tilde{\cW} = (t \cup t^\perp) \cap \cH(3, 4)$ and $\cW' \cap \tilde{\cW} = (u \cup u^\perp) \cap \cH(3, 4)$, where $t, u$ are secant to $\cE, \cE'$, respectively. Suppose by contradiction that $|t \cap (s \cup s^\perp)| = 0$, then $|t^\perp \cap (s \cup s^\perp)| = 0$. Moreover, in this case necessarily $|u \cap (s \cup s^\perp)| = |u^\perp \cap (s \cup s^\perp)| = 0$. Otherwise we would have that $|\cW \cap \tilde{\cW}| = |(t \cup t^\perp) \cap \cH(3, 4)| + |u \cap (s \cup s^\perp)| \ge 7$, which gives a contradiction. Hence, $(t \cup t^\perp) \cap \cH(3, 4) \subset \Sigma \setminus (s \cup s^\perp)$ and $(u \cup u^\perp) \cap \cH(3, 4) \subset \Sigma' \setminus (s \cup s^\perp)$. In particular, $t \cap \tilde{\Sigma}, t^\perp \cap \tilde{\Sigma}, u \cap \tilde{\Sigma}, u^\perp \cap \tilde{\Sigma}, s \cap \tilde{\Sigma}, s^\perp \cap \tilde{\Sigma}$ would be six pairwise disjoint lines of $\tilde{\Sigma}$, which gives a contradiction. It follows that $|t \cap (s \cup s^\perp)| \ge 1$ and $|u \cap (s \cup s^\perp)| \ge 1$. 

Firstly assume that $|s \cap \tilde{\Sigma}| = 3$. Then $t = u = s$ and $\tilde{\cW} = \cW''$. In this case, since $\cE \cap \tilde{\cE} \in \cW \cap \tilde{\cW}$ and $\cE' \cap \tilde{\cE} \in \cW' \cap \tilde{\cW}$, we have that necessarily $\tilde{\cE}$ is the unique elliptic ovoid of $\tilde{\cW}$ meeting $s$ in the points $R, T$. Therefore, we find one member of $\cV$ adjacent to both $\cE$ and $\cE'$ in this way. 

Secondly, let us assume that $|s \cap \tilde{\Sigma}| = 1$. Then $t \cap s = s \cap \tilde{\Sigma}$ and $t \cap s = u \cap s$, otherwise $|s \cap \tilde{\Sigma}| = 3$. Note that the nucleus of the conic $\cE' \setminus s$ is $R$ and the nucleus of $\cE \setminus s$ is $T$. Indeed, if $e$ is the unique external line of the Fano plane $\pi = \langle \cE' \setminus s \rangle \cap \Sigma'$ to the conic $\cE' \setminus s$, then $e^\perp = ST$, since $S, T$ are the only two points of $\cE'$ not on $\pi$ and therefore $e^\perp \cap \pi = \{R\}$. Suppose by contradiction that $u \cap s = t \cap s = \{R\}$. Then $u \cap \Sigma'$ is secant to $\cE' \setminus s$ and $R \in u$, a contradiction. A similar argument leads to $u \cap s = t \cap s \ne \{T\}$. Therefore, $u \cap s = t \cap s = \{S\}$. Let $\gamma$ be the Fano plane $\langle u, t \rangle \cap \tilde{\Sigma}$. Let $\ell$ be the line through $S$ distinct from $u$ and $t$ such that $|\ell \cap \gamma| = 3$. Since $\ell \cap \tilde{\Sigma}$ is tangent to $\tilde{\cE}$, we have that $\ell \cap \tilde{\Sigma}$, $\ell \cap \Sigma''$, $\ell \cap \Sigma'$, $\ell \cap \Sigma$ are generators of $\tilde{\cW}, \cW'', \cW', \cW$, respectively, and $\ell$ is a generator of $\cH(3, 4)$. Hence, $\ell \cap s^\perp$ is a point, say $Z$. In particular, $\tilde{\cW} \cap \cW'' = \ell \cap \tilde{\Sigma} = \ell \cap \Sigma''$. On the other hand, if $\ell$ is a line joining $S$ with a point $Z$ of $s^\perp$, then there is a unique symplectic subquadrangle $\tilde{\cW}$ of $\cH(3, 4)$ meeting $\cW''$ in the three points of $\ell \cap \Sigma''$. Using the Klein correspondence it can be seen that such subquadrangle meets both $\cW$ and $\cW'$ in the six points of $\cH(3, 4)$ that lie on two secant lines, one of them through $S$ and one of them through $Z$. One of the two secant lines through $S$ meets $\cE$ in two points and has a further point, say $Z_1$, while the other secant through $S$ intersects $\cE'$ in two points and has a further point, say $Z_2$. Since there is a unique elliptic ovoid of $\tilde{\cW}$ containing $Z_1, Z_2, S$ we have that there arise three members of $\cV$ adjacent to both $\cE$ and $\cE'$ in this way. 
\end{proof}

\begin{prop}\label{nonadj1}
Let $\cE, \cE' \in \cV$ and let $\cW, \cW'$ be their symplectic subquadrangle. If $|\cW \cap \cW'| = 3$ and $\cE$ and $\cE'$ are not adjacent, then there are eight elliptic quadrics of $\cV$ adjacent to both $\cE$ and $\cE'$.
\end{prop}
\begin{proof}
Suppose that $\cE, \cE'$ are two non--adjacent elliptic quadrics of $\cV$, where $\cW$, $\cW'$ have in common the three points $R, S, T$ of a generator $\ell$. Thus $|\cE \cap \cE'| \in \{0, 1\}$. Let $\tilde{\Sigma}$ be the Baer subgeometry containing a symplectic subquadrangle $\tilde{\cW}$ meeting both $\cW$, $\cW'$ in six points and let $\cW \cap \tilde{\cW} = (t \cup t^\perp) \cap \cH(3, 4)$ and $\cW' \cap \tilde{\cW} = (u \cup u^\perp) \cap \cH(3, 4)$. Let $\tilde{\cE}$ be an elliptic ovoid of a symplectic subquadrangle $\tilde{\cW}$ of $\cH(3, 4)$ adjacent to both $\cE, \cE'$. From Lemma \ref{lemma1} {\em i)}, $\tilde{\cW}$ can be chosen in twelve ways. Observe that the lines $t$, $u$ lie in the plane $\left(t^\perp \cap u^\perp \cap \ell\right)^\perp$ and $t^\perp$, $u^\perp$ in the plane $\left(t \cap u \cap \ell\right)^\perp$. Particularly, $\tilde{\cW}$ is uniquely determined by two points of $\Sigma \cap \Sigma'$ and two lines secant to $\cH(3, 4)$, namely $t \subset \cW, u \subset \cW'$, that pass through one of the two points and are contained in the tangent plane of the other. 

\medskip
\fbox{$\cE \cap \cE' = \{S\}$.}
\medskip 

If one of the two points of $\Sigma \cap \Sigma' \cap \tilde{\Sigma}$ is $S$, then without loss of generality we may assume that $t \cap u = \{S\}$. Then $t$, $u$ are secant to $\cE$, $\cE'$, respectively. Let $t \cap \Sigma = \{S, T_1, T_2\}$, $u \cap \Sigma' = \{S, U_1, U_2\}$, $t \cap \cE = \{S, T_1\}$, $u \cap \cE' = \{S, U_1\}$ and $t^\perp \cap u^\perp \cap \ell = \{R\}$. It follows that $R, S, U_1, U_2, T_1, T_2$ are six of the seven points of the Fano plane $\tilde{\Sigma} \cap R^\perp$ and the three generators through $R$ are $R T_1$, $R U_1$ and $\ell$. Suppose by contradiction that $U_1, T_1 \in \tilde{\cE}$. Thus it is not difficult to see that $S \in \tilde{\cE}$, which is in a contradiction with the fact that $|\cE \cap \tilde{\cE}| = |\cE' \cap \tilde{\cE}| = 1$. Therefore, we get that $\tilde{\cE}$ has to be the unique elliptic ovoid of $\tilde{\cW}$ containing the conic $S, T_2, U_2$. Among the twelve symplectic subquadrangles given above there are four intersecting $\Sigma \cap \Sigma'$ in $\{S, R\}$. Hence, there are four elliptic quadrics of $\cV$ that are adjacent to both $\cE$ and $\cE'$ arising in this way. 

If none of the two points of $\Sigma \cap \Sigma' \cap \tilde{\Sigma}$ coincides with $\cE \cap \cE'$, without loss of generality we may assume that $t \cap u = \{R\}$ and $t^\perp \cap u^\perp = \{T\}$. Since $T \not\in \cE \cup \cE'$, we have that $\Sigma \cap T^\perp \cap \cE$, $\Sigma' \cap T^\perp \cap \cE'$ are conics of $\cW$, $\cW'$, respectively. Therefore, if $\tilde{\cW}$ is one of the four subquadrangles intersecting $\Sigma \cap \Sigma'$ in $\{R, T\}$, then exactly one of the following possibilities occurs: $|t \cap \cE| = 2$ and $|u \cap \cE'| = 0$, $|t \cap \cE| = 0$ and $|u \cap \cE'| = 2$, $|t \cap \cE| = 2$ and $|u \cap \cE'| = 2$, $|t \cap \cE| = 0$ and $|u \cap \cE'| = 0$. In the first case since the unique point $\tilde{\cE} \cap \cE$ belongs to $t$, we have that $|t \cap \tilde{\cE}| = 2$ and $R \in \tilde{\cE}$. It follows that $|u \cap \tilde{\cE}| = 2$. Therefore, $|t^\perp \cap \tilde{\cE}| = |u^\perp \cap \tilde{\cE}| = 0$ and $|\cE' \cap \tilde{\cE}| = 0$, which gives a contradiction. In the second case, a similar argument holds true. In the third case, let $t \cap \Sigma = \{R, T_1, T_2\}$, $u \cap \Sigma' = \{R, U_1, U_2\}$, where the three generators through $T$ are the lines $T T_1 = T U_1$, $T T_2 = T U_2$, $\ell$. In this case $\tilde{\cE}$ has to contain either $T_1$ and $U_2$ or $T_2$ and $U_1$. Similarly if the fourth case occurs. Hence, there are four elliptic quadrics of $\cV$ that are adjacent to both $\cE$ and $\cE'$ arising in this way. 

\medskip
\fbox{$|\cE \cap \cE'| = 0$.} 
\medskip
\par\noindent
The elliptic ovoid $\cE$ of $\cW$ has a point in common with $\ell \cap \Sigma$, say $R$. Analogously $\ell \cap \Sigma' \cap \cE'$ is a point, say $T$. 

Let $t \cap u = \{R\}$ and $t^\perp \cap u^\perp = \{S\}$. Since $S \not\in \cE \cup \cE'$, we have that $\Sigma \cap S^\perp \cap \cE$, $\Sigma' \cap S^\perp \cap \cE'$ are conics of $\cW$, $\cW'$, respectively. Among the four subquadrangles intersecting $\Sigma \cap \Sigma'$ in $\{R, S\}$, there are two with the property that $|t \cap \cE| = 2$ and $|u \cap \cE'| = 0$, whereas the remaining two are such that $|t \cap \cE| = 2$ and $|u \cap \cE'| = 2$. In the former case, since $|\tilde{\cE} \cap \cE'| = 1$ and $\tilde{\cE} \cap \cE' \in u^\perp$, it follows that $S \in \tilde{\cE}$. Hence, $|t^\perp \cap \tilde{\cE}| = 2$, which in turn implies $|t \cap \tilde{\cE}| = 0$. Therefore, $|\cE \cap \tilde{\cE}| = 0$, which gives a contradiction. If the latter case occurs let $t \cap \Sigma = \{R, T_1, T_2\}$, $u \cap \Sigma' = \{R, U_1, U_2\}$, $u \cap \cE' = \{U_1, U_2\}$, $t \cap \cE = \{R, T_1\}$, where the three generators through $S$ are $S T_1 = S U_1$, $S U_2$, $\ell$. The unique elliptic ovoid of $\tilde{\cW}$ containing $R$ and $U_2$ or $U_2$ and $T_1$ has to contain $T_1$ or $R$. It follows that $\tilde{\cE}$ pass through $R$ and $U_1$ (hence $T_2 \in \tilde{\cE}$). A similar argument applies if $t \cap u = \{T\}$ and $t^\perp \cap u^\perp = \{S\}$. There arise four elliptic ovoids of $\cV$ adjacent to both $\cE$ and $\cE'$ in this way. 

Let $t \cap u = \{R\}$ and $t^\perp \cap u^\perp = \{T\}$. In this case $\Sigma \cap T^\perp \cap \cE$ is a conic of $\cW$ and $\Sigma' \cap T^\perp \cap \cE' = \{T\}$. Hence $|u \cap \cE'| = 0$, $t \cap \cE = \{R, Z\}$, $|t^\perp \cap \cE| = 0$ and $u^\perp \cap \cE' = \{T, Z'\}$. Note that $R T, R Z'$ and $T Z$ are generators, hence $\tilde{\cE}$ has to contain $Z$ and $Z'$. Since there are four symplectic subquadrangles meeting $\Sigma \cap \Sigma'$ in $\{R, T\}$, we get four elliptic ovoids of $\cV$ adjacent to both $\cE$ and $\cE'$ in this way.
\end{proof}

\begin{prop}\label{nonadj2}
Let $\cE, \cE' \in \cV$ and let $\cW, \cW'$ be their symplectic subquadrangle. If $|\cW \cap \cW'| = 6$ and $\cE$ and $\cE'$ are not adjacent, then there are eight elliptic quadrics of $\cV$ adjacent to both $\cE$ and $\cE'$.
\end{prop}
\begin{proof}
Let $\cE, \cE'$ be two non--adjacent elliptic quadrics of $\cV$, where $\cW$, $\cW'$ have in common the six points $(s \cup s^\perp) \cap \cH(3, 4)$. Thus $|\cE \cap \cE'| \in \{0, 2\}$. Let $\tilde{\Sigma}$ be the Baer subgeometry containing a symplectic subquadrangle $\tilde{\cW}$ meeting both $\cW$, $\cW'$ in six points and let $\cW \cap \tilde{\cW} = (t \cup t^\perp) \cap \cH(3, 4)$ and $\cW' \cap \tilde{\cW} = (u \cup u^\perp) \cap \cH(3, 4)$. Let $\tilde{\cE}$ be an elliptic ovoid of $\tilde{\cW}$ that is adjacent to both $\cE, \cE'$. From Lemma \ref{lemma1} {\em ii)}, $\tilde{\cW}$ can be chosen in ten ways. Let $s \cap \Sigma = \{R, S, T\}$ and $s^\perp \cap \Sigma = \{R', S', T'\}$. If $t = u = s$, then $\tilde{\cW} = \cW''$. In the remaining nine cases $|\tilde{\Sigma} \cap s| = |\tilde{\Sigma} \cap s^\perp| = 1$. We will assume that $t \cap u$ is a point of $s$ and $t^\perp \cap u^\perp$ is a point of $s^\perp$, i.e., $t \cap s = u \cap s = t \cap u$ and $t^\perp \cap s^\perp = u^\perp \cap s^\perp = t^\perp \cap u^\perp$. The lines $t$, $u$ lie in the plane $\left(t^\perp \cap u^\perp \cap s^\perp\right)^\perp$ and $t^\perp$, $u^\perp$ in the plane $\left(t \cap u \cap s\right)^\perp$ and hence the plane $\langle t, u \rangle$ contains $s$ and the plane $\langle t^\perp, u^\perp \rangle$ contains $s^\perp$. In particular, $\tilde{\cW}$ is uniquely determined by a point of $s$ and a point of $s^\perp$. 

\medskip
\fbox{$\cE \cap \cE' = \{R, S\}$.} 
\medskip

If $\tilde{\cW} = \cW''$, then an elliptic ovoid of $\tilde{\cW}$ that is adjacent to both $\cE$ and $\cE'$ contains two points of $s$ and hence either passes through $T$ and $R$ or through $T$ and $S$. Hence, there are two elliptic ovoids of $\tilde{\cW}$ adjacent to both $\cE$ and $\cE'$ in this case. 

Let $\tilde{\cW} \ne \cW''$. Assume that $t \cap s \notin \cE \cap \cE'$, i.e., $t \cap s = \{T\}$. Then $t, u$ are external to $\cE$, $\cE'$, respectively, otherwise the plane $\langle t, u \rangle \cap \Sigma$ or $\langle t, u \rangle \cap \Sigma'$ would have four points in common with $\cE$ or $\cE'$, which is not possible. Therefore $t^\perp$ and $u^\perp$ are secant to $\cE$ and $\cE'$, respectively. Let $t^\perp \cap \cE = \{T_1, T_2\}$ and $u^\perp \cap \cE' = \{U_1, U_2\}$, where $T T_1 = T U_1$ and $T U_2 = T T_2$ are generators. Then an elliptic ovoid of $\tilde{\cW}$ that is adjacent to both $\cE$ and $\cE'$ either contains $T_1$ and $U_2$ or $T_2$ and $U_1$. Since there are three subquadrangles $\tilde{\cW}$ meeting $s$ in $T$, there arise six elliptic quadrics of $\cV$ adjacent to both $\cE$ and $\cE'$ in this way. 

If $t \cap s \in \cE \cap \cE'$, then we assume without loss of generality that $t \cap s = \{R\}$. The lines $t, u$ are secant to $\cE$, $\cE'$, respectively. Let $t \cap \Sigma = \{R, P_1, Z_1\}$, $u \cap \Sigma' = \{R, P_2, Z_2\}$, $t \cap \cE = \{R, P_1\}$, $u \cap \cE' = \{R, P_2\}$. Hence, $\langle t, u \rangle \cap \Sigma \cap \cE = \{R, S, P_1\}$ and $\langle t, u \rangle \cap \Sigma' \cap \cE' = \{R, S, P_2\}$, where $\langle t, u \rangle = T'^\perp$. Note that necessarily the line $P_1 P_2 = T T'$ is a generator and $\tilde{\cE}$ does not contain both $P_1, P_2$. Therefore, $t \cap \tilde{\cE} = \{R, Z_1\}$ and $u \cap \tilde{\cE} = \{R, Z_2\}$, otherwise $|\cE \cap \tilde{\cE}| > 1$ or $|\cE' \cap \tilde{\cE}| > 1$. On the other hand, the three generators through $T'$ are the lines $T' R$, $T' P_1 = T' P_2$ and $T' Z_1 = T' Z_2$, contradicting the fact that $\tilde{\cE}$ is an ovoid of $\tilde{\cW}$.   

\medskip
\fbox{$|\cE \cap \cE'| = 0$.} 
\medskip
\par\noindent
In this case let $\cE \cap s = \{R, S\}$ and $\cE' \cap s^\perp = \{R' S'\}$.

If $\tilde{\cW} = \cW''$, and since $|\cE \cap \tilde{\cE}| = |\cE' \cap \tilde{\cE}| = 1$, we have that $\tilde{\cE}$ has to contain at least a point of $s$ and at least a point of $s^\perp$. On the other hand $\{s, s^\perp\}$ consists of one secant and one external line to $\tilde{\cE}$, which gives a contradiction. 

Let $\tilde{\cW} \ne \cW''$. If $t \cap s \notin \cE$, then $t \cap s = \{T\}$ and $t$ is external to $\cE$, otherwise $|\langle t, u \rangle \cap \Sigma \cap \cE| = 4$. Therefore, $t^\perp$ is secant to $\cE$. If $t^\perp \cap s^\perp \not\in \cE'$, since $T'^\perp = (t^\perp \cap s^\perp)^\perp = \langle u, t \rangle$ then the planes $\langle u, t \rangle \cap \Sigma$, $\langle u, t \rangle \cap \Sigma'$ are secant to $\cE$, $\cE'$, respectively. Moreover, the line $u \cap \Sigma'$ is secant to $\cE'$, indeed the three lines of $\langle u, t \rangle \cap \Sigma'$ through $T$ are $s \cap \Sigma'$, $u \cap \Sigma'$ and $T T' \cap \Sigma'$, where the latter is a generator of $\cW'$. Since $|\cE' \cap \tilde{\cE}| = 1$ and $\cW' \cap \tilde{\cW} = (u \cup u^\perp) \cap \cH(3, 4)$, the elliptic ovoid $\tilde{\cE}$ contains one point of $\cE' \cap u$ and hence $s$ is secant to $\tilde{\cE}$. It follows that $T \in \tilde{\cE}$ and $t$ is secant to $\tilde{\cE}$. Therefore, $t^\perp$ is external to $\tilde{\cE}$ and $|\cE \cap \tilde{\cE}| = |t^\perp \cap \cE \cap \tilde{\cE}| = 0$, which gives a contradiction. We may conclude that $t^\perp \cap s^\perp \in \cE'$. Without loss of generality let $t^\perp \cap s^\perp = \{R'\}$. In this case the plane $\langle u, t \rangle \cap \Sigma$ is secant to $\cE$ and the plane $\langle u, t \rangle \cap \Sigma'$ is tangent to $\cE'$ at $R'$. The lines $t, u$ are external to $\cE$, $\cE'$, respectively. Hence $t^\perp, u^\perp$ are secant to $\cE$, $\cE'$, respectively. Let $u^\perp \cap \cE' = \{R', P\}$ and $t^\perp \cap \cE = \{Z_1, Z_2\}$. Since $\langle t^\perp, u^\perp \rangle = T^\perp$, the lines $T R'$, $T Z_1 = T P$ and $T Z_2$ are the three generators through $T$. If $R', Z_2 \in \tilde{\cE}$, then $P \in \tilde{\cE}$ and $|\cE' \cap \tilde{\cE}| = 2$, a contradiction. Since $P, Z_2$ are not both in $\tilde{\cE}$, we have that $\tilde{\cE}$ has to be the elliptic ovoid of $\tilde{\cW}$ containing $R'$ and $Z_1$. Varying $t^\perp \cap s^\perp$ in $s^\perp \cap \cE'$ and interchanging the role of $s$ and $s^\perp$, we get four elliptic quadrics of $\cV$ adjacent to both $\cE$ and $\cE'$. 

If $t \cap s \in \cE$ and $t^\perp \cap s^\perp \in \cE'$, then we assume without loss of generality that $t \cap s = \{R\}$, $t^\perp \cap s^\perp = \{R'\}$. The line $t$ is secant to $\cE$ and $t^\perp$ is external to $\cE$, whereas $u^\perp$ is secant to $\cE'$ and $u$ is external to $\cE'$. Let $t \cap \cE = \{R, P\}$, $u^\perp \cap \cE' = \{R', P'\}$. Then the lines $R R'$, $T' R = T' P'$, $T R' = T P$ are generators. Note that the line $P P'$ is not a generator, otherwise $R', P' \in P^\perp$ and hence $P^\perp = \langle P, R', P' \rangle$. Since $R^\perp = \langle R, R', P' \rangle$, we would have $u = R' P' = P^\perp \cap R^\perp = P R = t$, a contradiction. Let $\tilde{\cE}$ be the elliptic ovoid of $\tilde{\cW}$ containing $P$ and $P'$. Then $t$ and $u^\perp$ are secant to $\tilde{\cE}$ and hence $t^\perp$ and $u$ are external to $\tilde{\cE}$. It follows that $|\cE \cap \tilde{\cE}| = |\cE' \cap \tilde{\cE}| = 1$, as required. Varying $t \cap s$ in $s \cap \cE$ and $t^\perp \cap s^\perp$ in $s^\perp \cap \cE'$, we obtain four elliptic quadrics of $\cV$ adjacent to both $\cE$ and $\cE'$. 
\end{proof}

\begin{theorem}
The graph $\Gamma$ is a strongly regular graph with parameters $(216, 40, 4, 8)$.
\end{theorem}
\begin{proof}
Firstly, let us we prove that $\Gamma$ is a $40$--regular graph. Let $\cE$ be an elliptic quadric of $\cV$ and let $\cW \subseteq \Sigma$ be its symplectic subquadrangle. Let $s$ be a line that is not a generator of $\cW$ such that $s \cap \Sigma = \{P, P_1, P_2\}$ and $s \cap \cE = \{P_1, P_2\}$. Let $\cW'$ and $\cW''$  be the subquadrangles of $\cH(3, 4)$, contained in the Baer subgeometries $\Sigma', \Sigma''$, such that $\Sigma \cap \Sigma' \cap \Sigma'' = \Sigma \cap (s \cup s^\perp)$. An elliptic ovoid $\cO$ of $\cW'$ or of $\cW''$ is adjacent to $\cE$ if and only if it contains the points $P$ and $P_i$, $i = 1,2$. Hence, there arise $4$ elliptic ovoids of $\cV$ that are adjacent with $\cE$. Varying the line $s$ among the $10$ secant lines of $\cE$, we have that there are $40$ elliptic ovoids of $\cV$ that are adjacent with $\cE$.

Let $\cE, \cE'$ be two elliptic quadrics of $\cV$ and let $\cW, \cW'$ be their symplectic subquadrangle. Assume that $\cE$ and $\cE'$ are adjacent. From Proposition \ref{adj}, there are four elliptic quadrics of $\cV$ adjacent to both $\cE$ and $\cE'$. 

Let us assume that $\cE, \cE'$ are not adjacent. We claim that there are eight members of $\cV$ that are adjacent to both $\cE$ and $\cE'$. Let $\tilde{\Sigma}$ be the Baer subgeometry containing a symplectic subquadrangle $\tilde{\cW}$ meeting both $\cW$, $\cW'$ in six points and let $\cW \cap \tilde{\cW} = (t \cup t^\perp) \cap \cH(3, 4)$ and $\cW' \cap \tilde{\cW} = (u \cup u^\perp) \cap \cH(3, 4)$. Let $\tilde{\cE}$ be an elliptic ovoid of a symplectic subquadrangle $\tilde{\cW}$ of $\cH(3, 4)$ adjacent to both $\cE, \cE'$. We distinguish several cases. If $|\cW \cap \cW'| \in \{3, 6\}$, we have that from Proposition \ref{nonadj1}  and Proposition \ref{nonadj2} there are eight elliptic quadrics of $\cV$ that are adjacent to both $\cE$ and $\cE'$. If $\cW = \cW'$ and $\cE \cap \cE' = \{S\}$, then $t = u$ and $S \in t$. Indeed, a line $\ell$ of $\Sigma$ that is secant to $\cE$ and does not contain $S$, has to be external to $\cE'$. Hence, $\ell^\perp$ is external to $\cE$ and secant to $\cE'$. Therefore, no elliptic ovoids of a symplectic subquadrangle of $\cH(3, 4)$ meeting $\cW$ in $(\ell \cup \ell^\perp) \cap \cH(3, 4)$ is adjacent to both $\cE$ and $\cE'$. On the other hand, if $t \cap \Sigma = \{R,S,T\}$, then there are two symplectic subquadrangles meeting $\cW$ in $(t \cup t^\perp) \cap \cH(3, 4)$ and each of them has exactly one elliptic ovoid containing $R$ and $T$. Since there are four lines through $S$ that are secant to both $\cE$ and $\cE'$, we have that there are eight members of $\cV$ that are adjacent to both $\cE$ and $\cE'$. 
\end{proof}

By construction, the vertices of the graph $\Gamma$ are permuted in a single orbit under the action of $\PSU(4, 2) \simeq O^-(6, 2) \simeq \PSp(4, 3)$ and hence the graph $\Gamma$ coincides with the strongly regular graph with parameters $(216, 40, 4, 8)$ constructed in \cite{CRS}. 

\begin{cor}
The graph $\Gamma$ admits $\PGaU(4,2)$ as an automorphism group.
\end{cor}

\section{The maximal cliques of $\Gamma$}

In this section we determine the maximal cliques of the graph $\Gamma$. Let $\cC$ be a maximal clique of $\Gamma$. From Proposition \ref{adj}, $|\cC| \ge 3$. Hence, $\cC$ contains three elliptic quadrics of $\cV$, say $\cE$, $\cE'$, $\cE''$. Let $\cW \subset \Sigma$, $\cW \subset \Sigma'$, $\cW'' \subset \Sigma''$ be their symplectic subquadrangles.

\begin{lemma}
A maximal clique of $\Gamma$ has size three.
\end{lemma}
\begin{proof}
It is enough to show that $|\cC| \le 3$. In the dual setting $\cW, \cW', \cW''$ correspond to three parabolic quadrics $\cP, \cP', \cP''$ of $\cQ^-(5, 2)$ pairwise intersecting in a three--dimensional hyperbolic quadric $\cQ^+(3, 2)$. Under the polarity associated with $\cQ^-(5, 2)$, the hyperplanes containing the parabolic quadrics $\cP$, $\cP'$, $\cP''$ are mapped to three points $P_1, P_2, P_3 \in \PG(5, 2) \setminus \cQ^-(5, 2)$ such that the line joining two of them is external to $\cQ^-(5, 2)$. We show that no further point $P$ of $\PG(5, 2) \setminus \cQ^-(5, 2)$ has the property that the lines $P P_i$, $i = 1,2,3$, are external to $\cQ^-(5, 2)$.

If $P_1, P_2, P_3$ are the points of a line $\ell$ external to $\cQ^-(5,2)$ and $P$ is a further point of $\PG(5, 2) \setminus \cQ^-(5, 2)$ such that the lines $P P_i$, $i = 1,2,3$, are external to $\cQ^-(5, 2)$, then the plane $\langle P, \ell \rangle$ has no point in common with $\cQ^-(5, 2)$, which gives a contradiction. In this case $|\cW \cap \cW' \cap \cW''| = 6$. 

Suppose that $\langle P_1, P_2, P_3 \rangle$ is a plane. In this case $\Sigma \cap \Sigma'$ share the six points of $(s \cup s^\perp) \cap \cH(3, 4)$. From the proof of Proposition \ref{adj}, if $s \cap \Sigma = \{R, S, T\}$, $\cE \cap s = \{R, S\}$ and $\cE' \cap s  = \{S, T\}$, then $\cE \cap \cE' = \cE \cap \cE'' = \cE' \cap \cE'' = \{S\}$ and $\cW''$ is one of the three symplectic subquadrangles meeting the line $s$ in the point $S$, the line $s^\perp$ in one of the three points of $s^\perp \cap \cH(3, 4)$ and intersecting both $\cW$, $\cW'$ in six points. In the dual setting, if $\cQ = \cP \cap \cP'$, the point $S$ correspond to a line of $\cQ$, say $r$, whereas the points of $s^\perp \cap \cH(3, 4)$ correspond to the three lines $\cQ$ intersecting $r$ in one point. The symplectic subquadrangle $\cW''$ correspond to one of the three parabolic quadric, say $\cP_j$, $j = 1,2,3$, meeting $\cQ$ in two intersecting lines one of which being $r$. Note that $\cP''$ coincides with $\cP_j$, for some $j$, and the hyperplanes containing the parabolic quadrics $\cP_j$, $1 \le j \le 3$, are mapped to three points $P_3, P_4, P_5 \in \PG(5, 2) \setminus \cQ^-(5, 2)$ such that the lines $P_i P_j$, $1 \le i \le 2$, $3 \le j \le 5$ are external to $\cQ^-(5, 2)$. Since the three hyperplanes containing $\cP_j$, $1 \le j \le 3$, meet in a plane $\pi$ through the line $r$, we have that $\langle P_3, P_4 , P_5 \rangle = \pi$. Hence, the line $P_j P_k$, $3 \le j < k \le 5$, is not external to $\cQ^-(5,2)$. In this case $|\cW \cap \cW' \cap \cW''| = 2$.   
\end{proof}

From the previous lemma and the proof of Proposition \ref{adj}, the following result follows.

\begin{prop}
The graph $\Gamma$ has $1440$ maximal cliques such that $|\cW \cap \cW' \cap \cW''| = 6$ and $4320$ maximal cliques such that $|\cW \cap \cW' \cap \cW''| = 2$.
\end{prop}


\noindent {\bf Acknowledgement} \\
D. Crnkovi\' c and A. \v Svob were supported by {\rm C}roatian Science Foundation under the project 6732.


\begin{thebibliography}{SK}

\bibitem{CRS} D. Crnkovi\' c, S. Rukavina, A. \v{S}vob, New strongly regular graphs from orthogonal groups $O^+(6,2)$ and $O^-(6,2)$, {\em Discrete Math.} 341 (2018), no. 10, 2723--2728.

\bibitem{H} J. W. P. Hirschfeld, {\em Finite projective spaces of three dimensions}, Oxford Mathematical Monographs. Oxford Science Publications. The Clarendon Press, Oxford University Press, New York, 1985.

\bibitem{PT} S. E. Payne, J. A. Thas, {\em Finite Generalized Quadrangles}, Res. Notes Math. 110, Pitman, 1984.
\end{thebibliography}
\end{document}